\documentclass[a4paper,12pt]{article}
\usepackage{parskip}
\usepackage[normalem]{ulem}
\usepackage{amsfonts,amsmath,amssymb,xcolor,hyperref,amsthm,graphicx,comment,enumerate,stmaryrd}
\usepackage{tikz}
\usetikzlibrary{positioning}

\title{On exponentially accurate approximation of a near the identity map by an autonomous flow}

\author{V.~Gelfreich$^1$ and A.~Vieiro$^{2,3}$\\[4pt]
$^1$ \small Mathematics Institute, University of Warwick, \small Coventry CV4 7AL, UK\\
\small{\tt v.gelfreich@warwick.ac.uk}\\[4pt]
$^2$ \small Departament de Matem\`atiques i Inform\`atica,\\ \small Universitat de Barcelona, Gran Via 585, 08007 Barcelona, Spain\\
$^3$ \small Centre de Recerca Matem\`atica (CRM),\\ \small Campus Bellaterra, 08193 Bellaterra, Spain\\
\small{\tt vieiro@maia.ub.es}
}

\providecommand{\keywords}[1]
{
  \small	
  \textbf{\textit{Keywords:}} #1
}

\newcommand{\N}{\mathbb{N}}
\newcommand{\R}{\mathbb{R}}
\newcommand{\C}{\mathbb{C}}

\newtheorem{thm}{Theorem}

\newtheorem{remark}[thm]{Remark}

\begin{document}

\maketitle	
\begin{abstract}
This paper contains a proof of a refined version of Neishtadt's theorem which states
that an analytic near-identity map can be approximated by the time-one map of
an autonomous flow with exponential accuracy. We  provide explicit 
expressions for the vector fields and give explicit bounds for the error terms.
\end{abstract}


In the  perturbation theory of dynamical systems, there are two parallel and almost equivalent directions: 
one studies rapidly oscillating vector fields, while the second one deals with near-identity maps. 
The theory for flows possesses a powerful collection of analytic tools based on averaging
\cite{BM1961,Verhulst2023}. On the other hand,  numerical studies and visualization  of dynamics
are often much easier to perform using maps.

In 1984 Neishtadt \cite{Neishtadt84} published a theorem which states that a tangent to the identity analytic family of maps can be embedded into a family of autonomous flows up to an exponentially small error. 
The results of that paper found a substantial number of applications in Ordinary Differential Equations and Dynamical Systems.
Neishtadt's proof is based on the classical averaging
for rapidly oscillating time-periodic flows and does not provide
explicit expressions for the vector fields in terms of the original map. 
Therefore, checking the accuracy of an approximation for an individual map becomes
difficult and finding an expression for the vector field impossible from the practical point of view.

In our paper we present a refined version of  Neishtadt's theorem which overcomes these limitations:
 we provide an explicit expression for the vector fields and establish explicit bounds for the error terms.
In contrast to original Neishtadt's theorem,  our version makes a statement about individual maps and the approximation error is explicitly controlled 
by the
ratio  $\delta/\varepsilon$, where
$\varepsilon$ characterises the distance to the identity in a complex
$\delta$-neighbourhood of the domain of the map.

Our proof is based on the  development of discrete averaging method
 \cite{GelfreichV18} and substantially simplifies analysis of dynamics by eliminating the need of embedding the map into a non-autonomous flow. As an additional bonus, the explicit nature of our construction opens potential for developing new numerical methods based on the discrete averaging for maps.

We consider an analytic (or real-analytic) map $f:D_0\to\C^n$ defined on a subset 
 $D_0\subset \mathbb C^n$ (or $\R^n$). We suppose that there is $\delta>0$
 such that the analytic continuation of $f$ 
 onto $D_{\delta}$, a complex  $\delta$-neighbourhood of $D_0$, is close to the identity
 map~$\xi$ and define
\begin{equation}
\varepsilon=\|f-\xi\|_{D_{\delta}}.
\end{equation}
We use the infinity norm for vectors
and supremum norms for functions. 
Let $m\in\N$ and define an interpolating vector field of order $m$,
\begin{equation}\label{Eq:interpol_VF}
	X_m(x)=\sum_{k=1}^m\frac{(-1)^{k-1}}{k}\Delta_k(x),
\end{equation}
where the finite differences are defined recursively
\begin{equation}\label{Eq:finite-diff}
	\Delta_0(x)=x,
	\qquad 
	\Delta_{k}(x)=\Delta_{k-1}(f(x))-\Delta_{k-1}(x)
	\quad\text{for  $k\ge 1$}.
\end{equation}
We  say that $X_m$ is obtained with the help of  {\em discrete averaging\/}
as $X_m$ is a weighted sum of  the map iterates $f^k(x)$ for $0\le k\le m$.
Indeed, it is not too difficult to check that 
\[
	\Delta_k(x)=\sum_{i=0}^k \binom{k}{i}(-1)^{k-i} f^i(x).
\]
Let $\Phi^t_X$ denote the flow defined by the differential equation $\dot x=X(x)$.

\goodbreak

\begin{thm}\label{Thm:alaNeishtadt}
If a map $f$ is analytic in $D_\delta$ and  $\varepsilon/\delta \le 1/6e$,
then 
  the interpolating vector field $X_m$
of order $2\le m\le M_\varepsilon+1$, 
where $M_\varepsilon=\frac{\delta}{6\mathrm e\varepsilon}$, is analytic in $D_{\delta/3}$,   $\|X_m\|_{D_{\delta/3}}\le 2\varepsilon$ and
\begin{equation}\label{Eq:Phi-f}
    \|\Phi^1_{X_m}-f \|_{D_0} 
  \le   
    3\varepsilon  \left( \frac{6 (m-1)\varepsilon}{\delta}\right)^{m} .
\end{equation}
Moreover, for $\displaystyle m = \left\lfloor M_\varepsilon
\right\rfloor+1$
\begin{equation}
\|\Phi^1_{X_{\mathrm{m}}}-f \|_{D_0}\le 
3\,\varepsilon \exp\left(- \delta /6\mathrm e\varepsilon\right).
\end{equation}
\end{thm}

\begin{proof}
We consider the map  $f$ as a member of the family 
\[
f_\mu=(1-\mu)\xi +\mu f
\]
where $\mu$ is a complex parameter.  
Obviously the  function $f_\mu$ is analytic in the same domain $D_\delta$ as the function $f$. Then $|\Delta_1(x)|=|f_\mu(x)-x|\le |\mu|\varepsilon$
for any $x\in D_\delta$ and any $\mu$. Let $\mu_1=\delta/\varepsilon$
and 
\[
\mu_m=\frac{2\delta}{3\varepsilon (m-1)}\qquad\text{for $m\ge 1$ .}
\]
If $|\mu|\le \mu_m$, then for  every $x_0\in D_{\delta/3}$
the first  iterates $x_k:=f_\mu^k(x_0)\in D_{\delta}$ 
and $|x_{k+1}-x_k|\le |\mu|\varepsilon $ provided $0\le k\le m-1$ .
The definition \eqref{Eq:finite-diff}
implies that 
\[
    \Delta_{k}(x_0)=\sum_{j=0}^{k-1} (-1)^{k-j-1} \binom{k-1}{j}\Delta_1(x_j).
\]
Since $\sum_{j=0}^{k}\binom{k}{j}=2^{k}$ we get
\[
	\left\| \Delta_{k}\right\|_{D_{\delta/3}}
	\le  
	2^{k-1}|\mu|\varepsilon\,.
\]
Since $|\Delta_k(x)|\le \|\Delta_{k-1}'\|\, |\mu|\varepsilon$
where the supremum norm is taken over $|\mu|<\mu_k$ and $|x-x_0|<\mu_k\varepsilon$,
we can check that  $\Delta_{k}(x_0)=O(\mu^{k})$.
Applying the maximum modulus principle (MMP)%
\footnote{We  use the following simple statement of Complex Analysis:
 if a function $g$ is an analytic function of $\mu$ bounded in an open disk   $|\mu|<r$ and $g^{(k)}(0)=0$ for $k=0,1,\ldots,m$,
 then  the maximum modulus principle implies that 
 $|g(\mu )|\le (|\mu|/r)^{m} \sup_{|\mu|< r} |g(\mu )|$.
 Of course, if the function  extends continuously onto the boundary of the disk,
 the supremum can be replaced by the maximum over $|\mu|=r$.}
in $\mu$ to each component of $\Delta_{k}(x_0)$, we get 
\[
\left\|\Delta_{k}\right\|_{D_{\delta/3}}
\le
2^{k-1}\mu_{k}\varepsilon \,
\left(\frac{|\mu|}{\mu_{k}}\right)^{k} \,. 
\]

Let $X_{m,\mu}$ be defined by \eqref{Eq:interpol_VF} with $f$ replaced by $f_\mu$.
Then  $X_{m,\mu}$ is analytic in $D_{\delta/3}$ for $|\mu|\le \mu_{m}$
and admits the following upper bound
\[
\|X_{m,\mu}\|_{D_{\delta/3}}
\le\sum_{k=1}^m \frac1{k} \left\|\Delta_k\right\|_{D_{\delta/2}}
\le \frac{\varepsilon \mu_m}{2} \sum_{k=1}^m 
\left( \frac{2|\mu|}{\mu_{m}} \right)^{k}
\le 
\frac{\varepsilon|\mu|}{1-\frac{2|\mu|}{\mu_m}}\,.
\]
Then  we get that for $|\mu| \leq \mu_{m}/4$
\[
\|X_{m,\mu}\|_{D_{\delta/3}} 
\le 
 2\varepsilon|\mu|\,. 
\]
For our range of $m$ we have
 $\mu_{m}\ge4$. Then  the domain of validity of the upper bound 
includes $\mu=1$ and we get  
\[
\|X_{m}\|_{D_{\delta/2}}\le 2 \varepsilon.
\]
We  also get that for  $|\mu|\le\mu_{m}/4$ and $m\geq 2$
\[
 \|X_{m,\mu}\|_{D_{\delta/3}}\le\frac{\varepsilon\mu_{m}}{2}=\frac{\delta}{3(m-1)}\le\frac{\delta}{3}
 \;.
\]
Then the orbit of the vector field $X_{m,\mu}$ with an initial condition  in $D_0$ remains in $D_{\delta/3}$ during one unit of time and
$$
\|\Phi^1_{X_{m,\mu}}-\xi\|_{D_0}\le \|X_{m,\mu}\|_{D_{\delta/3}}\le 2\varepsilon|\mu|\,.
$$
In order to apply arguments based on the MMP, we need to check that  $\Phi^1_{X_{m,\mu}}$ 
has the same Taylor polynomial of degree $m$ in $\mu$ as the map
 $f_\mu$. Proofs of similar claims can be found in \cite{GelfreichV18,GelfreichViero2024}. 
First we define an auxiliary vector field
 \[
 Y_{m,\mu}(x)=\sum_{k=1}^m\mu^k a_k(x)
 \]
 where $a_1=f-\xi$ and $a_k$ with $k\ge 2$ are defined recursively by
 \[
 a_k=-\sum_{j=2}^k\frac{1}{j!}\sum_{i_1+\dots+i_j=k}L_{a_{i_1}}\ldots L_{a_{i_j}}\xi
 \]
 where  differential operators $L_ag=a\cdot\nabla g$ act on a vector valued function $g$ component-wise. Expanding the time-$t$ map $\Phi^t_{Y_{m,\mu}}$
 in Taylor series in $t$ we get
 \begin{equation}\label{Eq:PhiY}
 \Phi^t_{Y_{m,\mu}}=\xi+t Y_{m,\mu}+\sum_{k=2}^m \frac{t^k}{k!}L_{Y_{m,\mu}}^k\xi+O((t\mu)^{m+1}).
\end{equation}
Our choice of $a_k$ implies that the terms of order  $\mu^k$ cancel each other for $k=2,\ldots,m$ when $t=1$:
\[
 \Phi^1_{Y_{m,\mu}}=\xi+ Y_{m,\mu}+\sum_{k=2}^m \frac{1}{k!}L_{Y_{m,\mu}}^k\xi+O(\mu^{m+1})
 =\xi+\mu a_1+ O(\mu^{m+1})=f_\mu+O(\mu^{m+1}).
 \]
Iterating the map we get that $\Phi^k_{Y_{m,\mu}}=f_\mu^k+O(\mu^{m+1})$.
Using the equation \eqref{Eq:interpol_VF} with $f$ replaced by $\Phi^1_{Y_{m,\mu}}$
we obtain a vector field $\hat X_{m,\mu}=X_{m,\mu}+O(\mu^{m+1})$. 
We note that  $\hat X_{m,\mu}$ is the derivative at $t=0$ of the Newton
 interpolating polynomial of degree $m$ defined by the points $\Phi^t_{Y_{m,\mu}}$ with $t=0,1,\ldots, m$. Since the interpolation is exact on polynomials of degree $m$,
 the equation \eqref{Eq:PhiY} implies that $\hat X_{m,\mu}=Y_{m,\mu}
 +O(\mu^{m+1})$. Combining these two estimates we get that $X_{m,\mu}=Y_{m,\mu}
 +O(\mu^{m+1})$, i.e.,  $Y_{m,\mu}$ is the Taylor
 polynomial of degree $m$ in $\mu$ for the vector field $X_{m,\mu}$.
 Since the time-one map of a vector field depends smoothly on the
 vector field we conclude that
 \[
 \Phi^1_{X_{m,\mu}}=\Phi^1_{Y_{m,\mu}}+O(\mu^{m+1})=f_\mu+O(\mu^{m+1}).
 \]
Therefore the Taylor expansion in $\mu$ of $\Phi^1_{X_{m,\mu}}$ matches the Taylor expansion of $f_\mu$ up to the order $m$. 

Since $\|\xi-f_\mu \|_{D_0} = |\mu|\, \| \xi - f \|_{D_0} \leq |\mu| \varepsilon$, 
we get that
$$
\|\Phi^1_{X_{m,\mu}}-f_\mu \|_{D_0}\le \|\Phi^1_{X_{m,\mu}}-\xi\|_{D_0}+\|\xi-f_\mu \|_{D_0} \le 
3\varepsilon|\mu|.
$$
The MMP  based on the bound in the disk 
$|\mu|\le \mu_m/4$ can be applied with $\mu=1$ to get the desired estimate
\[
\|\Phi^1_{X_{m}}-f \|_{D_0}  
\le   3\varepsilon\,\left( \frac{ 4}{\mu_{m}}\right)^m 
=
 3\varepsilon\,\left( 
 \frac{6\varepsilon (m-1)}{\delta}\right)^{m} 
.
\]
The right-hand side depends on $m$ and takes the smallest values 
 near  $ M_\varepsilon$. There is a unique integer $m\in[M_\varepsilon,M_\varepsilon+1)$.
Then, for this $m$,
\[
\frac{\mu_{m}}{4}=\frac{\delta}{6\varepsilon (m-1)}\ge
\frac{\delta}{6\varepsilon M_\varepsilon}= \mathrm e>1.
\]
In particular, it satisfies the assumption used in the proof, and
we can conclude that
$$
\|\Phi^1_{X_m}-f \|_{D_0} 
\le 3 \varepsilon \, \mathrm e^{-M_\varepsilon} 
= 3  \,
\varepsilon
\exp
\left(-
\frac{\delta}{6\mathrm e \varepsilon}
\right).
$$
Theorem is proved.
\end{proof}

\medskip

\begin{remark} 
For the sake of completeness  we present the bounds for the case of  $m=1$ separately. 
The interpolating vector field is given by $X_1(x)=f(x)-x$ and 
\[
\|\Phi^1_{X_{1}}-f \|_{D_0}  
\le   \frac{2\varepsilon^2 }{\delta}
.
\]
\begin{proof}
In order to check this bound we can consider $X_{1,\mu}=f_\mu -\xi=\mu(f-\xi)$. 
Obviously,  \[\|X_{1,\mu}\|_{D_{\delta}}=\|\mu(f_\mu-\xi)\|_{D_{\delta}}= |\mu|\varepsilon.\]
Then 
$
\|\Phi^1_{X_{1,\mu}}-f_\mu \|_{D_0}\le 2|\mu|\varepsilon
$
provided $|\mu|\varepsilon<\delta$. Since $\Phi^1_{X_{1,\mu}}-f_\mu=O(\mu^2)$, the MMP implies
the desired bound
\[
\|\Phi^1_{X_{1}}-f \|_{D_0}  
\le   \frac{  2\varepsilon\mu_0}{\mu_0^{2}}
=
 \frac{2\varepsilon^2 }{\delta}
\]
where $\mu_0=\delta/\varepsilon $.
\end{proof}
\end{remark}

The error bounds of  Theorem~\ref{Thm:alaNeishtadt} can be  improved by implementing a symmetric interpolation scheme instead of the Newton one, in a way similar to \cite{GelfreichV18}. We also note that in the case of a symplectic map $f$,  
the interpolating 
vector field \eqref{Eq:interpol_VF} is typically not Hamiltonian. 
On the other hand, it can be shown to be a small perturbation 
of a Hamiltonian vector field \cite{GelfreichV18,GelfreichViero2024},
with the size of the perturbation being comparable with the approximation errors
of Theorem~\ref{Thm:alaNeishtadt}.

Finally, we note that the discrete averaging can also be used
to study near-the-identity families of finitely smooth maps.
More precisely, let $f=f_\varepsilon$ be a member of a
$C^{m+1}$-smooth family of maps such that  $f_0=\mathrm{id}$. Then 
the equation \eqref{Eq:interpol_VF} 
provides an explicit expression for a vector field
such that the corresponding time-one flow 
approximates  $f_\varepsilon$ 
up to an error of the order of $O(\varepsilon^{m+1})$.

\section*{Acknowledgements}
A.V. is supported by the Spanish grant PID2021-125535NB-I00 funded by 
MICIU/AEI/10.13039/501100011033 and by ERDF/EU.
He also acknowledges
the Catalan grant 2021-SGR-01072 and the
Severo Ochoa and Mar\'{\i}a de Maeztu Program for Centers and Units of
Excellence in R\&D (CEX2020-0010 84-M).

\bibliographystyle{plain}

\end{document}